\documentclass[12pt, a4paper, reqno]{amsart}
\usepackage{a4wide}
\usepackage{enumerate}

\newcommand{\shin}[1]{}
\newcommand{\seo}[1]{}
\makeatletter
\def\markboth#1#2{%
  \begingroup
    \@temptokena{{#1}{#2}}\xdef\@themark{\the\@temptokena}%
    \mark{\the\@temptokena}%
  \endgroup
  \if@nobreak\ifvmode\nobreak\fi\fi}
\def\thanks#1{\g@addto@macro\thankses{\thanks{#1}}}
\makeatother

\usepackage{color}
\usepackage{ifpdf}
\ifpdf 
  \usepackage[pdftex]{graphicx}
  \DeclareGraphicsExtensions{.pdf,.png,.mps}
  \usepackage{pgf}
  \usepackage{tikz}
\else 
  \usepackage{graphicx}
  \DeclareGraphicsExtensions{.eps,.bmp}
  \usepackage{pgf}
  \usepackage{tikz}
  \usepackage{pstricks}
\fi
\usepackage{epic,bez123}
\usepackage{floatflt}
\usepackage{wrapfig}

\usepackage{amsthm}

\theoremstyle{plain}
\newtheorem{thm}{Theorem}

\newtheorem{lem}[thm]{Lemma}

\theoremstyle{definition}

\newtheorem*{rmk}{Remark}

\usepackage{amsmath}
\DeclareMathOperator{\MD}{MD}

\newcommand{\abs}[1]{\left|#1\right|}
\newcommand{\set}[1]{\left\{#1\right\}}

\newcommand{\Tp}{\mathcal{T}^{(p)}}

\newcommand{\Y}{\mathcal{Y}^{(p)}}
\newcommand{\F}{\mathcal{F}^{(p)}}

\title{A refined enumeration of $p$-ary labeled trees}
\author{Seunghyun Seo}
\address[Seunghyun Seo]{Department of Mathematics Education, Kangwon National University, Chuncheon 200-701, Korea}
\email{shyunseo@kangwon.ac.kr}

\author{Heesung Shin$^\dag$}
\address[Heesung Shin]{Department of Mathematics, Inha University, Incheon 402-751, Korea}
\email{shin@inha.ac.kr}
\thanks{\dag Corresponding author}

\keywords{$p$-ary labeled tree, Refinement, Maximal decreasing subtree}
\subjclass[2010]{05A15, 05C05, 05C30}

\date{\today}

\begin{document}
\maketitle
\begin{abstract}
Let $\Tp_n$ be the set of $p$-ary labeled trees on $\{1,2,\dots,n\}$. A maximal decreasing subtree of an $p$-ary labeled tree
is defined by the maximal $p$-ary subtree from the root with all edges being decreasing. In this paper, we study a new refinement $\Tp_{n,k}$ of $\Tp_n$, which is the set of $p$-ary labeled trees whose maximal decreasing subtree has $k$ vertices.
\end{abstract}

\section{Introduction}

Let $p$ be a fixed integer greater than $1$. A \emph{$p$-ary tree} $T$ is a tree such that:
\begin{enumerate}
\item Either $T$ is empty or has a distinguished vertex $r$ which is called the root of $T$, and
\item $T-r$ consists of a weak ordered partition $(T_1,\dots,T_p)$ of $p$-ary trees.
\end{enumerate}
A $2$-ary(resp.~$3$-ary) tree is called binary(resp.~ternary) tree.
Figure~\ref{fig:ternary} exhibits all the ternary tree with $3$ vertices.
A {\em full $p$-ary tree} is a $p$-ary tree, where each vertex has either $0$ or $p$ children.
It is well known (see \cite[6.2.2 Proposition]{Sta99}) that the number of full
$p$-ary trees with $n$ internal vertices is given by the $n$th order-$p$ Fuss-Catalan number \cite[p.~361]{GKP89} $C_n^{(p)} = \frac{1}{pn+1} \binom{pn+1}{n}$.\seo{참고문헌추가}
Clearly a full $p$-ary tree $T$ with $m$ internal vertices corresponds to a $p$-ary tree  with $m$ vertices by deleting all the leaves in $T$, so the number of $p$-ary trees with $n$ vertices is also $C_n^{(p)}$.

\begin{figure}[h]
\centering
\ifpdf
\begin{pgfpicture}{17.00mm}{17.00mm}{158.00mm}{44.00mm}
\pgfsetxvec{\pgfpoint{1.00mm}{0mm}}
\pgfsetyvec{\pgfpoint{0mm}{1.00mm}}
\color[rgb]{0,0,0}\pgfsetlinewidth{0.30mm}\pgfsetdash{}{0mm}
\pgfmoveto{\pgfxy(25.00,40.00)}\pgflineto{\pgfxy(20.00,30.00)}\pgfstroke
\pgfmoveto{\pgfxy(25.00,40.00)}\pgflineto{\pgfxy(25.00,30.00)}\pgfstroke
\pgfmoveto{\pgfxy(35.00,40.00)}\pgflineto{\pgfxy(30.00,30.00)}\pgfstroke
\pgfmoveto{\pgfxy(35.00,40.00)}\pgflineto{\pgfxy(40.00,30.00)}\pgfstroke
\pgfmoveto{\pgfxy(45.00,40.00)}\pgflineto{\pgfxy(45.00,30.00)}\pgfstroke
\pgfmoveto{\pgfxy(45.00,40.00)}\pgflineto{\pgfxy(50.00,30.00)}\pgfstroke
\pgfmoveto{\pgfxy(65.00,40.00)}\pgflineto{\pgfxy(60.00,30.00)}\pgfstroke
\pgfmoveto{\pgfxy(60.00,30.00)}\pgflineto{\pgfxy(55.00,20.00)}\pgfstroke
\pgfmoveto{\pgfxy(62.00,33.00)}\pgflineto{\pgfxy(62.00,33.00)}\pgfstroke
\pgfmoveto{\pgfxy(75.00,40.00)}\pgflineto{\pgfxy(70.00,30.00)}\pgfstroke
\pgfmoveto{\pgfxy(70.00,30.00)}\pgflineto{\pgfxy(70.00,20.00)}\pgfstroke
\pgfmoveto{\pgfxy(85.00,40.00)}\pgflineto{\pgfxy(80.00,30.00)}\pgfstroke
\pgfmoveto{\pgfxy(80.00,30.00)}\pgflineto{\pgfxy(85.00,20.00)}\pgfstroke
\pgfmoveto{\pgfxy(95.00,40.00)}\pgflineto{\pgfxy(95.00,30.00)}\pgfstroke
\pgfmoveto{\pgfxy(95.00,30.00)}\pgflineto{\pgfxy(90.00,20.00)}\pgfstroke
\pgfmoveto{\pgfxy(105.00,40.00)}\pgflineto{\pgfxy(105.00,30.00)}\pgfstroke
\pgfmoveto{\pgfxy(105.00,30.00)}\pgflineto{\pgfxy(105.00,20.00)}\pgfstroke
\pgfmoveto{\pgfxy(115.00,40.00)}\pgflineto{\pgfxy(115.00,30.00)}\pgfstroke
\pgfmoveto{\pgfxy(115.00,30.00)}\pgflineto{\pgfxy(120.00,20.00)}\pgfstroke
\pgfmoveto{\pgfxy(125.00,40.00)}\pgflineto{\pgfxy(130.00,30.00)}\pgfstroke
\pgfmoveto{\pgfxy(130.00,30.00)}\pgflineto{\pgfxy(125.00,20.00)}\pgfstroke
\pgfmoveto{\pgfxy(135.00,40.00)}\pgflineto{\pgfxy(140.00,30.00)}\pgfstroke
\pgfmoveto{\pgfxy(140.00,30.00)}\pgflineto{\pgfxy(140.00,20.00)}\pgfstroke
\pgfmoveto{\pgfxy(145.00,40.00)}\pgflineto{\pgfxy(150.00,30.00)}\pgfstroke
\pgfmoveto{\pgfxy(150.00,30.00)}\pgflineto{\pgfxy(155.00,20.00)}\pgfstroke
\pgfcircle[fill]{\pgfxy(25.00,40.00)}{1.00mm}
\pgfcircle[stroke]{\pgfxy(25.00,40.00)}{1.00mm}
\pgfcircle[fill]{\pgfxy(20.00,30.00)}{1.00mm}
\pgfcircle[stroke]{\pgfxy(20.00,30.00)}{1.00mm}
\pgfcircle[fill]{\pgfxy(25.00,30.00)}{1.00mm}
\pgfcircle[stroke]{\pgfxy(25.00,30.00)}{1.00mm}
\pgfcircle[fill]{\pgfxy(30.00,30.00)}{1.00mm}
\pgfcircle[stroke]{\pgfxy(30.00,30.00)}{1.00mm}
\pgfcircle[fill]{\pgfxy(35.00,40.00)}{1.00mm}
\pgfcircle[stroke]{\pgfxy(35.00,40.00)}{1.00mm}
\pgfcircle[fill]{\pgfxy(45.00,40.00)}{1.00mm}
\pgfcircle[stroke]{\pgfxy(45.00,40.00)}{1.00mm}
\pgfcircle[fill]{\pgfxy(40.00,30.00)}{1.00mm}
\pgfcircle[stroke]{\pgfxy(40.00,30.00)}{1.00mm}
\pgfcircle[fill]{\pgfxy(45.00,30.00)}{1.00mm}
\pgfcircle[stroke]{\pgfxy(45.00,30.00)}{1.00mm}
\pgfcircle[fill]{\pgfxy(50.00,30.00)}{1.00mm}
\pgfcircle[stroke]{\pgfxy(50.00,30.00)}{1.00mm}
\pgfcircle[fill]{\pgfxy(55.00,20.00)}{1.00mm}
\pgfcircle[stroke]{\pgfxy(55.00,20.00)}{1.00mm}
\pgfcircle[fill]{\pgfxy(60.00,30.00)}{1.00mm}
\pgfcircle[stroke]{\pgfxy(60.00,30.00)}{1.00mm}
\pgfcircle[fill]{\pgfxy(65.00,40.00)}{1.00mm}
\pgfcircle[stroke]{\pgfxy(65.00,40.00)}{1.00mm}
\pgfcircle[fill]{\pgfxy(75.00,40.00)}{1.00mm}
\pgfcircle[stroke]{\pgfxy(75.00,40.00)}{1.00mm}
\pgfcircle[fill]{\pgfxy(70.00,30.00)}{1.00mm}
\pgfcircle[stroke]{\pgfxy(70.00,30.00)}{1.00mm}
\pgfcircle[fill]{\pgfxy(70.00,20.00)}{1.00mm}
\pgfcircle[stroke]{\pgfxy(70.00,20.00)}{1.00mm}
\pgfcircle[fill]{\pgfxy(85.00,40.00)}{1.00mm}
\pgfcircle[stroke]{\pgfxy(85.00,40.00)}{1.00mm}
\pgfcircle[fill]{\pgfxy(80.00,30.00)}{1.00mm}
\pgfcircle[stroke]{\pgfxy(80.00,30.00)}{1.00mm}
\pgfcircle[fill]{\pgfxy(85.00,20.00)}{1.00mm}
\pgfcircle[stroke]{\pgfxy(85.00,20.00)}{1.00mm}
\pgfcircle[fill]{\pgfxy(95.00,40.00)}{1.00mm}
\pgfcircle[stroke]{\pgfxy(95.00,40.00)}{1.00mm}
\pgfcircle[fill]{\pgfxy(95.00,30.00)}{1.00mm}
\pgfcircle[stroke]{\pgfxy(95.00,30.00)}{1.00mm}
\pgfcircle[fill]{\pgfxy(90.00,20.00)}{1.00mm}
\pgfcircle[stroke]{\pgfxy(90.00,20.00)}{1.00mm}
\pgfcircle[fill]{\pgfxy(105.00,40.00)}{1.00mm}
\pgfcircle[stroke]{\pgfxy(105.00,40.00)}{1.00mm}
\pgfcircle[fill]{\pgfxy(105.00,30.00)}{1.00mm}
\pgfcircle[stroke]{\pgfxy(105.00,30.00)}{1.00mm}
\pgfcircle[fill]{\pgfxy(105.00,20.00)}{1.00mm}
\pgfcircle[stroke]{\pgfxy(105.00,20.00)}{1.00mm}
\pgfcircle[fill]{\pgfxy(115.00,40.00)}{1.00mm}
\pgfcircle[stroke]{\pgfxy(115.00,40.00)}{1.00mm}
\pgfcircle[fill]{\pgfxy(115.00,30.00)}{1.00mm}
\pgfcircle[stroke]{\pgfxy(115.00,30.00)}{1.00mm}
\pgfcircle[fill]{\pgfxy(120.00,20.00)}{1.00mm}
\pgfcircle[stroke]{\pgfxy(120.00,20.00)}{1.00mm}
\pgfcircle[fill]{\pgfxy(125.00,40.00)}{1.00mm}
\pgfcircle[stroke]{\pgfxy(125.00,40.00)}{1.00mm}
\pgfcircle[fill]{\pgfxy(130.00,30.00)}{1.00mm}
\pgfcircle[stroke]{\pgfxy(130.00,30.00)}{1.00mm}
\pgfcircle[fill]{\pgfxy(125.00,20.00)}{1.00mm}
\pgfcircle[stroke]{\pgfxy(125.00,20.00)}{1.00mm}
\pgfcircle[fill]{\pgfxy(135.00,40.00)}{1.00mm}
\pgfcircle[stroke]{\pgfxy(135.00,40.00)}{1.00mm}
\pgfcircle[fill]{\pgfxy(140.00,30.00)}{1.00mm}
\pgfcircle[stroke]{\pgfxy(140.00,30.00)}{1.00mm}
\pgfcircle[fill]{\pgfxy(140.00,20.00)}{1.00mm}
\pgfcircle[stroke]{\pgfxy(140.00,20.00)}{1.00mm}
\pgfcircle[fill]{\pgfxy(145.00,40.00)}{1.00mm}
\pgfcircle[stroke]{\pgfxy(145.00,40.00)}{1.00mm}
\pgfcircle[fill]{\pgfxy(150.00,30.00)}{1.00mm}
\pgfcircle[stroke]{\pgfxy(150.00,30.00)}{1.00mm}
\pgfcircle[fill]{\pgfxy(155.00,20.00)}{1.00mm}
\pgfcircle[stroke]{\pgfxy(155.00,20.00)}{1.00mm}
\pgfcircle[stroke]{\pgfxy(25.00,40.00)}{2.00mm}
\pgfcircle[stroke]{\pgfxy(35.00,40.00)}{2.00mm}
\pgfcircle[stroke]{\pgfxy(45.00,40.00)}{2.00mm}
\pgfcircle[stroke]{\pgfxy(65.00,40.00)}{2.00mm}
\pgfcircle[stroke]{\pgfxy(75.00,40.00)}{2.00mm}
\pgfcircle[stroke]{\pgfxy(85.00,40.00)}{2.00mm}
\pgfcircle[stroke]{\pgfxy(95.00,40.00)}{2.00mm}
\pgfcircle[stroke]{\pgfxy(105.00,40.00)}{2.00mm}
\pgfcircle[stroke]{\pgfxy(115.00,40.00)}{2.00mm}
\pgfcircle[stroke]{\pgfxy(125.00,40.00)}{2.00mm}
\pgfcircle[stroke]{\pgfxy(135.00,40.00)}{2.00mm}
\pgfcircle[stroke]{\pgfxy(145.00,40.00)}{2.00mm}
\end{pgfpicture}%
\else
  \setlength{\unitlength}{1bp}%
  \begin{picture}(399.69, 76.54)(0,0)
  \put(0,0){\includegraphics{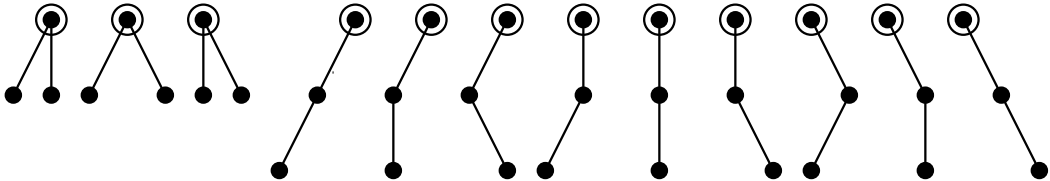}}
  \end{picture}%
\fi
\caption{All 12 ternary trees with $3$ vertices}
\seo{동그란 원으로 root를 명시하면 어떨지.. All ternary 사이에 12를 넣었음.}
\shin{root 넣었습니다.}
\label{fig:ternary}
\end{figure}

An {\em $p$-ary labeled tree} is a $p$-ary tree whose vertices are labeled by distinct positive integers. In most cases, a $p$-ary labeled tree with $n$ vertices is identified with an $p$-ary tree on the vertex set $[n]:=\{1,2,\dots,n\}$.
Let $\Tp_n$ be the set of $p$-ary labeled trees on $[n]$. Clearly the cardinality of $\Tp_n$ is given by
\begin{equation} \label{eq:pcat}
|\Tp_n| = n!\, C_n^{(p)} = (pn)_{(n-1)},
\end{equation}
where $m_{(k)}:=m(m-1)\cdots(m-k+1)$ is a falling factorial.

For a given $p$-ary labeled tree $T$, a \emph{maximal decreasing subtree} of $T$ is defined by the maximal $p$-ary subtree from the root with all edges being decreasing, denoted by $\MD(T)$. Figure~\ref{fig:tree} illustrates the maximal decreasing subtree of a given ternary tree $T$. Let $\Tp_{n,k}$ be the set of $p$-ary labeled trees on $[n]$ with its maximal decreasing subtree having $k$ vertices.
\begin{figure}[t]
\centering
\ifpdf
\begin{pgfpicture}{17.00mm}{16.14mm}{126.00mm}{54.14mm}
\pgfsetxvec{\pgfpoint{1.00mm}{0mm}}
\pgfsetyvec{\pgfpoint{0mm}{1.00mm}}
\color[rgb]{0,0,0}\pgfsetlinewidth{0.30mm}\pgfsetdash{}{0mm}
\pgfcircle[fill]{\pgfxy(40.00,50.00)}{1.00mm}
\pgfcircle[stroke]{\pgfxy(40.00,50.00)}{1.00mm}
\pgfcircle[fill]{\pgfxy(30.00,40.00)}{1.00mm}
\pgfcircle[stroke]{\pgfxy(30.00,40.00)}{1.00mm}
\pgfcircle[fill]{\pgfxy(40.00,40.00)}{1.00mm}
\pgfcircle[stroke]{\pgfxy(40.00,40.00)}{1.00mm}
\pgfcircle[fill]{\pgfxy(50.00,40.00)}{1.00mm}
\pgfcircle[stroke]{\pgfxy(50.00,40.00)}{1.00mm}
\pgfcircle[fill]{\pgfxy(50.00,30.00)}{1.00mm}
\pgfcircle[stroke]{\pgfxy(50.00,30.00)}{1.00mm}
\pgfcircle[fill]{\pgfxy(60.00,30.00)}{1.00mm}
\pgfcircle[stroke]{\pgfxy(60.00,30.00)}{1.00mm}
\pgfcircle[fill]{\pgfxy(30.00,30.00)}{1.00mm}
\pgfcircle[stroke]{\pgfxy(30.00,30.00)}{1.00mm}
\pgfcircle[fill]{\pgfxy(50.00,20.00)}{1.00mm}
\pgfcircle[stroke]{\pgfxy(50.00,20.00)}{1.00mm}
\pgfcircle[fill]{\pgfxy(70.00,20.00)}{1.00mm}
\pgfcircle[stroke]{\pgfxy(70.00,20.00)}{1.00mm}
\pgfcircle[fill]{\pgfxy(30.00,20.00)}{1.00mm}
\pgfcircle[stroke]{\pgfxy(30.00,20.00)}{1.00mm}
\pgfcircle[fill]{\pgfxy(20.00,20.00)}{1.00mm}
\pgfcircle[stroke]{\pgfxy(20.00,20.00)}{1.00mm}
\pgfmoveto{\pgfxy(40.00,50.00)}\pgflineto{\pgfxy(40.00,40.00)}\pgfstroke
\pgfmoveto{\pgfxy(40.00,50.00)}\pgflineto{\pgfxy(50.00,40.00)}\pgfstroke
\pgfmoveto{\pgfxy(50.00,40.00)}\pgflineto{\pgfxy(60.00,30.00)}\pgfstroke
\pgfmoveto{\pgfxy(60.00,30.00)}\pgflineto{\pgfxy(70.00,20.00)}\pgfstroke
\pgfmoveto{\pgfxy(60.00,30.00)}\pgflineto{\pgfxy(50.00,20.00)}\pgfstroke
\pgfmoveto{\pgfxy(50.00,40.00)}\pgflineto{\pgfxy(50.00,30.00)}\pgfstroke
\pgfmoveto{\pgfxy(40.00,50.00)}\pgflineto{\pgfxy(30.00,40.00)}\pgfstroke
\pgfmoveto{\pgfxy(30.00,40.00)}\pgflineto{\pgfxy(30.00,30.00)}\pgfstroke
\pgfmoveto{\pgfxy(30.00,30.00)}\pgflineto{\pgfxy(30.00,20.00)}\pgfstroke
\pgfmoveto{\pgfxy(30.00,30.00)}\pgflineto{\pgfxy(20.00,20.00)}\pgfstroke
\pgfputat{\pgfxy(32.00,39.00)}{\pgfbox[bottom,left]{\fontsize{11.38}{13.66}\selectfont 8}}
\pgfputat{\pgfxy(22.00,19.00)}{\pgfbox[bottom,left]{\fontsize{11.38}{13.66}\selectfont 7}}
\pgfputat{\pgfxy(52.00,19.00)}{\pgfbox[bottom,left]{\fontsize{11.38}{13.66}\selectfont 11}}
\pgfputat{\pgfxy(52.00,29.00)}{\pgfbox[bottom,left]{\fontsize{11.38}{13.66}\selectfont 3}}
\pgfputat{\pgfxy(32.00,19.00)}{\pgfbox[bottom,left]{\fontsize{11.38}{13.66}\selectfont 10}}
\pgfputat{\pgfxy(72.00,19.00)}{\pgfbox[bottom,left]{\fontsize{11.38}{13.66}\selectfont 5}}
\pgfputat{\pgfxy(52.00,39.00)}{\pgfbox[bottom,left]{\fontsize{11.38}{13.66}\selectfont 4}}
\pgfputat{\pgfxy(62.00,29.00)}{\pgfbox[bottom,left]{\fontsize{11.38}{13.66}\selectfont 6}}
\pgfputat{\pgfxy(43.00,49.00)}{\pgfbox[bottom,left]{\fontsize{11.38}{13.66}\selectfont 9}}
\pgfputat{\pgfxy(32.00,29.00)}{\pgfbox[bottom,left]{\fontsize{11.38}{13.66}\selectfont 1}}
\pgfputat{\pgfxy(42.00,39.00)}{\pgfbox[bottom,left]{\fontsize{11.38}{13.66}\selectfont 2}}
\pgfputat{\pgfxy(20.00,49.00)}{\pgfbox[bottom,left]{\fontsize{11.38}{13.66}\selectfont $T$}}
\pgfputat{\pgfxy(90.00,49.00)}{\pgfbox[bottom,left]{\fontsize{11.38}{13.66}\selectfont $\MD(T)$}}
\pgfsetlinewidth{1.20mm}\pgfmoveto{\pgfxy(75.00,35.00)}\pgflineto{\pgfxy(85.00,35.00)}\pgfstroke
\pgfmoveto{\pgfxy(85.00,35.00)}\pgflineto{\pgfxy(82.20,35.70)}\pgflineto{\pgfxy(82.20,34.30)}\pgflineto{\pgfxy(85.00,35.00)}\pgfclosepath\pgffill
\pgfmoveto{\pgfxy(85.00,35.00)}\pgflineto{\pgfxy(82.20,35.70)}\pgflineto{\pgfxy(82.20,34.30)}\pgflineto{\pgfxy(85.00,35.00)}\pgfclosepath\pgfstroke
\pgfcircle[fill]{\pgfxy(110.00,50.00)}{1.00mm}
\pgfsetlinewidth{0.30mm}\pgfcircle[stroke]{\pgfxy(110.00,50.00)}{1.00mm}
\pgfcircle[fill]{\pgfxy(100.00,40.00)}{1.00mm}
\pgfcircle[stroke]{\pgfxy(100.00,40.00)}{1.00mm}
\pgfcircle[fill]{\pgfxy(110.00,40.00)}{1.00mm}
\pgfcircle[stroke]{\pgfxy(110.00,40.00)}{1.00mm}
\pgfcircle[fill]{\pgfxy(120.00,40.00)}{1.00mm}
\pgfcircle[stroke]{\pgfxy(120.00,40.00)}{1.00mm}
\pgfcircle[fill]{\pgfxy(120.00,30.00)}{1.00mm}
\pgfcircle[stroke]{\pgfxy(120.00,30.00)}{1.00mm}
\pgfcircle[fill]{\pgfxy(100.00,30.00)}{1.00mm}
\pgfcircle[stroke]{\pgfxy(100.00,30.00)}{1.00mm}
\pgfmoveto{\pgfxy(110.00,50.00)}\pgflineto{\pgfxy(110.00,40.00)}\pgfstroke
\pgfmoveto{\pgfxy(110.00,50.00)}\pgflineto{\pgfxy(120.00,40.00)}\pgfstroke
\pgfmoveto{\pgfxy(120.00,40.00)}\pgflineto{\pgfxy(120.00,30.00)}\pgfstroke
\pgfmoveto{\pgfxy(110.00,50.00)}\pgflineto{\pgfxy(100.00,40.00)}\pgfstroke
\pgfmoveto{\pgfxy(100.00,40.00)}\pgflineto{\pgfxy(100.00,30.00)}\pgfstroke
\pgfputat{\pgfxy(102.00,39.00)}{\pgfbox[bottom,left]{\fontsize{11.38}{13.66}\selectfont 8}}
\pgfputat{\pgfxy(122.00,29.00)}{\pgfbox[bottom,left]{\fontsize{11.38}{13.66}\selectfont 3}}
\pgfputat{\pgfxy(122.00,39.00)}{\pgfbox[bottom,left]{\fontsize{11.38}{13.66}\selectfont 4}}
\pgfputat{\pgfxy(113.00,49.00)}{\pgfbox[bottom,left]{\fontsize{11.38}{13.66}\selectfont 9}}
\pgfputat{\pgfxy(102.00,29.00)}{\pgfbox[bottom,left]{\fontsize{11.38}{13.66}\selectfont 1}}
\pgfputat{\pgfxy(112.00,39.00)}{\pgfbox[bottom,left]{\fontsize{11.38}{13.66}\selectfont 2}}
\pgfcircle[stroke]{\pgfxy(40.00,50.00)}{2.00mm}
\pgfcircle[stroke]{\pgfxy(110.00,50.00)}{2.00mm}
\end{pgfpicture}%
\else
  \setlength{\unitlength}{1bp}%
  \begin{picture}(308.98, 107.72)(0,0)
  \put(0,0){\includegraphics{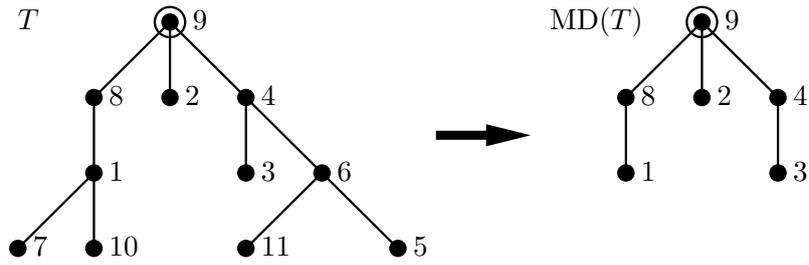}}
  \put(42.52,64.81){\fontsize{11.38}{13.66}\selectfont 8}
  \put(14.17,8.12){\fontsize{11.38}{13.66}\selectfont 7}
  \put(99.21,8.12){\fontsize{11.38}{13.66}\selectfont 11}
  \put(99.21,36.46){\fontsize{11.38}{13.66}\selectfont 3}
  \put(42.52,8.12){\fontsize{11.38}{13.66}\selectfont 10}
  \put(155.91,8.12){\fontsize{11.38}{13.66}\selectfont 5}
  \put(99.21,64.81){\fontsize{11.38}{13.66}\selectfont 4}
  \put(127.56,36.46){\fontsize{11.38}{13.66}\selectfont 6}
  \put(73.70,93.16){\fontsize{11.38}{13.66}\selectfont 9}
  \put(42.52,36.46){\fontsize{11.38}{13.66}\selectfont 1}
  \put(70.87,64.81){\fontsize{11.38}{13.66}\selectfont 2}
  \put(8.50,93.16){\fontsize{11.38}{13.66}\selectfont $T$}
  \put(206.93,93.16){\fontsize{11.38}{13.66}\selectfont $\MD(T)$}
  \put(240.94,64.81){\fontsize{11.38}{13.66}\selectfont 8}
  \put(297.64,36.46){\fontsize{11.38}{13.66}\selectfont 3}
  \put(297.64,64.81){\fontsize{11.38}{13.66}\selectfont 4}
  \put(272.13,93.16){\fontsize{11.38}{13.66}\selectfont 9}
  \put(240.94,36.46){\fontsize{11.38}{13.66}\selectfont 1}
  \put(269.29,64.81){\fontsize{11.38}{13.66}\selectfont 2}
  \end{picture}%
\fi
\caption{The maximal decreasing subtree of the ternary labeled tree $T$}
\seo{label $0$을 $11$로 바꿀 것. 첫번째 그림 코멘트와 같은 맥락으로 root 표시}
\shin{바꾸었습니다.}
\label{fig:tree}
\end{figure}

In this paper we present a formula for $|\Tp_{n,k}|$, which makes a refined enumeration of $\Tp_n$, or a generalization of equation~\eqref{eq:pcat}. Note that a similar refinement for rooted labeled trees and ordered labeled trees were done before (see \cite{SS12a, SS12b}), but the $p$-ary case is much more complicated and has quite different features.

\section{Main results}

From now on we will consider labeled trees only. So we will omit the word ``labeled".
Recall that $\Tp_{n,k}$ is the set of $p$-ary trees on $[n]$ with its maximal decreasing ordered subtree having $k$ vertices. Let $\Y_{n,k}$ be the set of $p$-ary trees $T$ on $[n]$, where $T$ is given by attaching additional $(n-k)$ increasing leaves to a decreasing tree with $k$ vertices. Let $\F_{n,k}$ be the set of (non-ordered) forests on $[n]$ consisting of $k$
$p$-ary trees, where the $k$ roots are not ordered. In Figure~\ref{fig:forest}, the first two forests are the same, but the third one is a different forest in $\mathcal{F}^{(2)}_{4,2}$.

\begin{figure}[h]
\centering
\ifpdf
\begin{pgfpicture}{-4.00mm}{-3.86mm}{111.00mm}{14.14mm}
\pgfsetxvec{\pgfpoint{1.00mm}{0mm}}
\pgfsetyvec{\pgfpoint{0mm}{1.00mm}}
\color[rgb]{0,0,0}\pgfsetlinewidth{0.30mm}\pgfsetdash{}{0mm}
\pgfcircle[fill]{\pgfxy(0.00,0.00)}{1.00mm}
\pgfcircle[stroke]{\pgfxy(0.00,0.00)}{1.00mm}
\pgfcircle[fill]{\pgfxy(10.00,0.00)}{1.00mm}
\pgfcircle[stroke]{\pgfxy(10.00,0.00)}{1.00mm}
\pgfcircle[fill]{\pgfxy(0.00,10.00)}{1.00mm}
\pgfcircle[stroke]{\pgfxy(0.00,10.00)}{1.00mm}
\pgfcircle[fill]{\pgfxy(15.00,10.00)}{1.00mm}
\pgfcircle[stroke]{\pgfxy(15.00,10.00)}{1.00mm}
\pgfmoveto{\pgfxy(0.00,10.00)}\pgflineto{\pgfxy(0.00,0.00)}\pgfstroke
\pgfmoveto{\pgfxy(0.00,10.00)}\pgflineto{\pgfxy(10.00,0.00)}\pgfstroke
\pgfcircle[stroke]{\pgfxy(0.00,10.00)}{2.00mm}
\pgfcircle[stroke]{\pgfxy(15.00,10.00)}{2.00mm}
\pgfcircle[fill]{\pgfxy(50.00,0.00)}{1.00mm}
\pgfcircle[stroke]{\pgfxy(50.00,0.00)}{1.00mm}
\pgfcircle[fill]{\pgfxy(60.00,0.00)}{1.00mm}
\pgfcircle[stroke]{\pgfxy(60.00,0.00)}{1.00mm}
\pgfcircle[fill]{\pgfxy(50.00,10.00)}{1.00mm}
\pgfcircle[stroke]{\pgfxy(50.00,10.00)}{1.00mm}
\pgfmoveto{\pgfxy(50.00,10.00)}\pgflineto{\pgfxy(50.00,0.00)}\pgfstroke
\pgfmoveto{\pgfxy(51.00,10.00)}\pgflineto{\pgfxy(60.00,0.00)}\pgfstroke
\pgfcircle[stroke]{\pgfxy(50.00,10.00)}{2.00mm}
\pgfcircle[fill]{\pgfxy(35.00,10.00)}{1.00mm}
\pgfcircle[stroke]{\pgfxy(35.00,10.00)}{1.00mm}
\pgfcircle[stroke]{\pgfxy(35.00,10.00)}{2.00mm}
\pgfcircle[fill]{\pgfxy(95.00,0.00)}{1.00mm}
\pgfcircle[stroke]{\pgfxy(95.00,0.00)}{1.00mm}
\pgfcircle[fill]{\pgfxy(105.00,0.00)}{1.00mm}
\pgfcircle[stroke]{\pgfxy(105.00,0.00)}{1.00mm}
\pgfcircle[fill]{\pgfxy(95.00,10.00)}{1.00mm}
\pgfcircle[stroke]{\pgfxy(95.00,10.00)}{1.00mm}
\pgfmoveto{\pgfxy(95.00,10.00)}\pgflineto{\pgfxy(95.00,0.00)}\pgfstroke
\pgfmoveto{\pgfxy(95.00,10.00)}\pgflineto{\pgfxy(105.00,0.00)}\pgfstroke
\pgfcircle[stroke]{\pgfxy(95.00,10.00)}{2.00mm}
\pgfcircle[fill]{\pgfxy(80.00,10.00)}{1.00mm}
\pgfcircle[stroke]{\pgfxy(80.00,10.00)}{1.00mm}
\pgfcircle[stroke]{\pgfxy(80.00,10.00)}{2.00mm}
\pgfputat{\pgfxy(13.00,-1.00)}{\pgfbox[bottom,left]{\fontsize{11.38}{13.66}\selectfont \makebox[0pt]{$1$}}}
\pgfputat{\pgfxy(63.00,-1.00)}{\pgfbox[bottom,left]{\fontsize{11.38}{13.66}\selectfont \makebox[0pt]{$1$}}}
\pgfputat{\pgfxy(92.00,-1.00)}{\pgfbox[bottom,left]{\fontsize{11.38}{13.66}\selectfont \makebox[0pt]{$1$}}}
\pgfputat{\pgfxy(19.00,9.00)}{\pgfbox[bottom,left]{\fontsize{11.38}{13.66}\selectfont \makebox[0pt]{$3$}}}
\pgfputat{\pgfxy(39.00,9.00)}{\pgfbox[bottom,left]{\fontsize{11.38}{13.66}\selectfont \makebox[0pt]{$3$}}}
\pgfputat{\pgfxy(84.00,9.00)}{\pgfbox[bottom,left]{\fontsize{11.38}{13.66}\selectfont \makebox[0pt]{$3$}}}
\pgfputat{\pgfxy(4.00,9.00)}{\pgfbox[bottom,left]{\fontsize{11.38}{13.66}\selectfont \makebox[0pt]{$2$}}}
\pgfputat{\pgfxy(54.00,9.00)}{\pgfbox[bottom,left]{\fontsize{11.38}{13.66}\selectfont \makebox[0pt]{$2$}}}
\pgfputat{\pgfxy(99.00,9.00)}{\pgfbox[bottom,left]{\fontsize{11.38}{13.66}\selectfont \makebox[0pt]{$2$}}}
\pgfputat{\pgfxy(3.00,-1.00)}{\pgfbox[bottom,left]{\fontsize{11.38}{13.66}\selectfont \makebox[0pt]{$4$}}}
\pgfputat{\pgfxy(53.00,-1.00)}{\pgfbox[bottom,left]{\fontsize{11.38}{13.66}\selectfont \makebox[0pt]{$4$}}}
\pgfputat{\pgfxy(108.00,-1.00)}{\pgfbox[bottom,left]{\fontsize{11.38}{13.66}\selectfont \makebox[0pt]{$4$}}}
\pgfputat{\pgfxy(25.00,4.00)}{\pgfbox[bottom,left]{\fontsize{11.38}{13.66}\selectfont \makebox[0pt]{$=$}}}
\pgfputat{\pgfxy(70.00,4.00)}{\pgfbox[bottom,left]{\fontsize{11.38}{13.66}\selectfont \makebox[0pt]{$\neq$}}}
\end{pgfpicture}%
\else
  \setlength{\unitlength}{1bp}%
  \begin{picture}(325.98, 51.02)(0,0)
  \put(0,0){\includegraphics{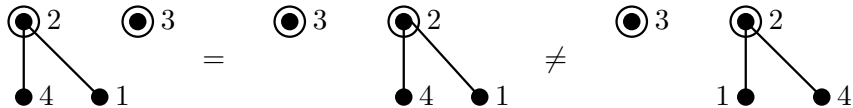}}
  \put(48.19,8.12){\fontsize{11.38}{13.66}\selectfont \makebox[0pt]{$1$}}
  \put(189.92,8.12){\fontsize{11.38}{13.66}\selectfont \makebox[0pt]{$1$}}
  \put(272.13,8.12){\fontsize{11.38}{13.66}\selectfont \makebox[0pt]{$1$}}
  \put(65.20,36.46){\fontsize{11.38}{13.66}\selectfont \makebox[0pt]{$3$}}
  \put(121.89,36.46){\fontsize{11.38}{13.66}\selectfont \makebox[0pt]{$3$}}
  \put(249.45,36.46){\fontsize{11.38}{13.66}\selectfont \makebox[0pt]{$3$}}
  \put(22.68,36.46){\fontsize{11.38}{13.66}\selectfont \makebox[0pt]{$2$}}
  \put(164.41,36.46){\fontsize{11.38}{13.66}\selectfont \makebox[0pt]{$2$}}
  \put(291.97,36.46){\fontsize{11.38}{13.66}\selectfont \makebox[0pt]{$2$}}
  \put(19.84,8.12){\fontsize{11.38}{13.66}\selectfont \makebox[0pt]{$4$}}
  \put(161.57,8.12){\fontsize{11.38}{13.66}\selectfont \makebox[0pt]{$4$}}
  \put(317.48,8.12){\fontsize{11.38}{13.66}\selectfont \makebox[0pt]{$4$}}
  \put(82.20,22.29){\fontsize{11.38}{13.66}\selectfont \makebox[0pt]{$=$}}
  \put(209.76,22.29){\fontsize{11.38}{13.66}\selectfont \makebox[0pt]{$\neq$}}
  \end{picture}%
\fi
\caption{Forests in $\mathcal{F}^{(3)}_{4,2}$}
\seo{두번째 세번째 forest는 component가 너무 붙어있는 듯. 3과 2 사이에 좀 더 간격을..}
\shin{적용했습니다.}
\label{fig:forest}
\end{figure}

Define the numbers
\begin{align*}
t(n,k) &= \left|\Tp_{n,k}\right|,\\
y(n,k) &= \left|\Y_{n,k}\right|,\\
f(n,k) &= \left|\F_{n,k}\right|.
\end{align*}
We will show that a $p$-ary tree can be ``decomposed" into a $p$-ary tree in $\bigcup_{n,k}\Y_{n,k}$ and a~forest in $\bigcup_{n,k}\F_{n,k}$. Thus it is important to count the numbers $y(n,k)$ and $f(n,k)$.

\begin{lem} \label{lem:znk}
For $n \ge 2$, the number $y(n,k)$ satisfies the recursion:
\begin{align}\label{eq:recurZ}
y(n,k) =
\sum_{m=0}^p \binom{n-1}{m} \binom{p}{m}m! \left((k-1)p-n+m+2\right)\,y(n-m-1,k-1)
\quad \text{for $1\le k <n$}
\end{align}
with the following boundary conditions:
\begin{align}
y(n,n) &= \prod_{j=0}^{n-1}(1+(p-1)j) \qquad \text{for $n \ge 1$} \label{eq:Z3}\\
y(n,k) &=0 \qquad \text{for $k < \max\left(\frac{n-1}{p},1\right)$} \label{eq:Z2}.
\end{align}
\end{lem}

\begin{proof}
Consider a tree $Y$ in $\Y_{n,k}$. The tree $Y$ with $n$ vertices consists of its maximal decreasing tree with $k$ vertices and the number of increasing leaves is $n-k$.
Note that the vertex $1$ is always contained in $\MD(Y)$.

If the vertex $1$ is a leaf of $Y$, consider the tree $Y'$ by deleting the leaf $1$ from $Y$.
The number of vertices in $Y'$ and $\MD(Y')$ are $n-1$ and $k-1$, respectively.
So the number of possible trees $Y'$ is $y(n-1,k-1)$.
Since we cannot attach the vertex $1$ to $n-k$ increasing leaves of $Y'$, there are $(k-1)p-(n-2)$ ways of recovering $Y$.
Thus the number of $Y$ with the leaf $1$ is
\begin{equation} \label{eq:1-leaf}
((k-1)p-n+2) \cdot y(n-1,k-1).
\end{equation}

If the vertex $1$ is not a leaf of $Y$, then the vertex $1$ has increasing leaves $\ell_1,\dots,\ell_m$, where $1\le m\le p$.
Consider the tree $Y''$ obtained by deleting $\ell_1,\dots,\ell_m$ from $Y$. Clearly $1$ is a leaf of $Y''$ and the number of vertices in $Y''$ and $\MD(Y'')$ are $n-m$ and $k$, respectively.
Thus by \eqref{eq:1-leaf}, the number of possible trees $Y''$ is $((k-1)p-(n-m)+2) \cdot y(n-m-1,k-1)$.
To recover $Y$ is to relabel all the vertices except $1$ of $Y''$ with the label set $\set{2,3,\dots,n} \setminus \set {\ell_1,\dots,\ell_m}$ and to attach the leaves $\ell_1,\dots,\ell_m$ to the vertex $1$ of $Y''$.
 Clearly ${\ell_1,\dots,\ell_m}$ is a subset of $\set{2,3,\dots,n}$. It is obvious that a way of attaching $\ell_1,\dots,\ell_m$ to vertex $1$ can be regarded as an injection from ${\ell_1,\dots,\ell_m}$ to $[p]$.
Thus the number of $Y$ without the leaf $1$ is
\begin{equation} \label{eq:1-nonleaf}
\binom{n-1}{m} \binom{p}{m}m! ((k-1)p-(n-m)+2) \cdot y(n-m-1,k-1)\,.
\end{equation}
Since $m$ may be the number from $1$ to $p$ and substituting $m=0$ in \eqref{eq:1-nonleaf} yields \eqref{eq:1-leaf}, we have the recursion \eqref{eq:recurZ}.

Since $\Y_{n,n}$ is the set of decreasing $p$-ary trees on $[n]$, the equation \eqref{eq:Z3} holds (see~\cite{BFS92}).
If the inequality $pk-(k-1)<n-k$ holds, $\Y_{n,k}$ should be empty. For $n\ge 1 $ and $k=0$, $\Y_{n,k}$ is also empty. Thus the equation \eqref{eq:Z2} also holds.
\end{proof}
The table for $y(n,k)$ with $p=2$ is shown in Table~\ref{table:ynk}.
\begin{table}[t]
{\footnotesize
\begin{tabular}{c|rrrrrrrrrrr}
$n\backslash k$ &0&1&2&3&4&5&6&7&8&9&10\\ \hline
0&1&&&&&&&&&&\\
1&0&1&&&&&&&&&\\
2&0&2&2&&&&&&&&\\
3&0&2&10&6&&&&&&&\\
4&0&0&24&56&24&&&&&&\\
5&0&0&24&256&360&120&&&&&\\
6&0&0&0&640&2672&2640&720&&&&\\
7&0&0&0&720&11824&28896&21840&5040&&&\\
8&0&0&0&0&30464&196352&330624&201600&40320&&\\
9&0&0&0&0&35840&857728&3177600&4032000&2056320&362880&\\
10&0&0&0&0&0&2251008&20640512&51901440&52496640&22982400&3628800\\
\end{tabular}
}
\vskip 1em
\caption{$y(n,k)$ with $p=2$}
\label{table:ynk}
\seo{table과 caption이 너무 붙어있는 느낌.. upper triangular 부분의 0은 생략해도 좋을 듯.}
\shin{빈 줄 넣음.}
\end{table}


Now we calculate $f(n,k)$ which is the number of forests on $[n]$ consisting of $k$ $p$-ary trees, where the $k$ components are not ordered. Here we use the convention that the empty product is~$1$.
\begin{lem} \label{lem:fnk}
For $0\le k\le n$, we have
\begin{align}
f(n,k) = {n \choose k} pk \prod_{i=1}^{n-k-1}(pn-i) \quad\text{if $n>k$},  \label{eq:num_f}
\end{align}
else $f(n,n) = 1$.
\end{lem}

\begin{proof}
Consider a forest $F$ in $\F_{n,k}$.
The forest $F$ consists of (non-ordered) $p$-ary trees $T_1,\ldots,T_k$ with roots $r_1, r_2, \ldots, r_k$, where $r_1 < r_2 < \cdots < r_k$.
The number of ways for choosing roots $r_1, r_2, \cdots, r_k$ from $[n]$ is equal to $n \choose k$.
From the \emph{reverse Pr\"ufer algorithm (RP Algorithm)} in \cite{SS07}, the number of ways for adding $n-k$ vertices successively to $k$ roots $r_1, r_2, \cdots, r_k$ is equal to
$$
pk (pn-1) (pn-2) \cdots (pn-n+k+1)
$$
for $ 0<  k < n$, thus the equation \eqref{eq:num_f} holds. For $0=k<n$, $\F_{n,0}$ is empty, so $f(n,0)=0$ included in \eqref{eq:num_f}.
For $0 \le k=n$,  $\F_{n,n}$ is the set of forests with no edges, so $f(n,n)=1$.
\end{proof}

Since the number $y(n,k)$ is determined by the recurrence relation \eqref{eq:recurZ} in Lemma~\ref{lem:znk}, we can count the number $t(n,k)$ with the
following theorem.

\begin{thm} \label{thm:main}
For $n \ge 1$, we have
\begin{align}
t(n,k) = \sum_{m=k}^{n} \binom{n}{m} \,\frac{m-k}{n-k} (pn-pk)_{(n-m)}\, y(m,k) \quad \text{if $1 \le k<n$,} \label{eq:num_o}
\end{align}
else $t(n,n) = \prod_{j=0}^{n-1}\,(pj-j+1)$, where $a_{(\ell)}:=a(a-1)\cdots(a-\ell+1)$ is a falling factorial.
\end{thm}

\begin{proof}
Given a $p$-ary tree $T$ in $\Tp_{n,k}$, let $Y$ be the subtree of $T$ consisting of $\MD(T)$ and its increasing leaves.
If $Y$ has $m$ vertices, then $Y$ is a subtree of $T$ with $(m-k)$ increasing leaves.
Also, the induced subgraph $Z$ of $T$ generated by the $(n-k)$ vertices not belonging to $\MD(T)$ is a (non-ordered) forest consisting of $(m-k)$ $p$-ary trees whose roots are increasing leaves of $Y$.
Figure~\ref{fig:TYZ} illustrates the subgraph $Y$ and $Z$ of a given ternary tree $T$.
\begin{figure}[t]
\centering
\ifpdf
\begin{pgfpicture}{17.00mm}{-7.86mm}{151.00mm}{54.14mm}
\pgfsetxvec{\pgfpoint{1.00mm}{0mm}}
\pgfsetyvec{\pgfpoint{0mm}{1.00mm}}
\color[rgb]{0,0,0}\pgfsetlinewidth{0.30mm}\pgfsetdash{}{0mm}
\pgfcircle[fill]{\pgfxy(40.00,50.00)}{1.00mm}
\pgfcircle[stroke]{\pgfxy(40.00,50.00)}{1.00mm}
\pgfcircle[fill]{\pgfxy(30.00,40.00)}{1.00mm}
\pgfcircle[stroke]{\pgfxy(30.00,40.00)}{1.00mm}
\pgfcircle[fill]{\pgfxy(40.00,40.00)}{1.00mm}
\pgfcircle[stroke]{\pgfxy(40.00,40.00)}{1.00mm}
\pgfcircle[fill]{\pgfxy(50.00,40.00)}{1.00mm}
\pgfcircle[stroke]{\pgfxy(50.00,40.00)}{1.00mm}
\pgfcircle[fill]{\pgfxy(50.00,30.00)}{1.00mm}
\pgfcircle[stroke]{\pgfxy(50.00,30.00)}{1.00mm}
\pgfcircle[fill]{\pgfxy(60.00,30.00)}{1.00mm}
\pgfcircle[stroke]{\pgfxy(60.00,30.00)}{1.00mm}
\pgfcircle[fill]{\pgfxy(30.00,30.00)}{1.00mm}
\pgfcircle[stroke]{\pgfxy(30.00,30.00)}{1.00mm}
\pgfcircle[fill]{\pgfxy(50.00,20.00)}{1.00mm}
\pgfcircle[stroke]{\pgfxy(50.00,20.00)}{1.00mm}
\pgfcircle[fill]{\pgfxy(70.00,20.00)}{1.00mm}
\pgfcircle[stroke]{\pgfxy(70.00,20.00)}{1.00mm}
\pgfcircle[fill]{\pgfxy(30.00,20.00)}{1.00mm}
\pgfcircle[stroke]{\pgfxy(30.00,20.00)}{1.00mm}
\pgfcircle[fill]{\pgfxy(20.00,20.00)}{1.00mm}
\pgfcircle[stroke]{\pgfxy(20.00,20.00)}{1.00mm}
\pgfmoveto{\pgfxy(40.00,50.00)}\pgflineto{\pgfxy(40.00,40.00)}\pgfstroke
\pgfmoveto{\pgfxy(40.00,50.00)}\pgflineto{\pgfxy(50.00,40.00)}\pgfstroke
\pgfmoveto{\pgfxy(50.00,40.00)}\pgflineto{\pgfxy(60.00,30.00)}\pgfstroke
\pgfmoveto{\pgfxy(60.00,30.00)}\pgflineto{\pgfxy(70.00,20.00)}\pgfstroke
\pgfmoveto{\pgfxy(60.00,30.00)}\pgflineto{\pgfxy(50.00,20.00)}\pgfstroke
\pgfmoveto{\pgfxy(50.00,40.00)}\pgflineto{\pgfxy(50.00,30.00)}\pgfstroke
\pgfmoveto{\pgfxy(40.00,50.00)}\pgflineto{\pgfxy(30.00,40.00)}\pgfstroke
\pgfmoveto{\pgfxy(30.00,40.00)}\pgflineto{\pgfxy(30.00,30.00)}\pgfstroke
\pgfmoveto{\pgfxy(30.00,30.00)}\pgflineto{\pgfxy(30.00,20.00)}\pgfstroke
\pgfmoveto{\pgfxy(30.00,30.00)}\pgflineto{\pgfxy(20.00,20.00)}\pgfstroke
\pgfputat{\pgfxy(32.00,39.00)}{\pgfbox[bottom,left]{\fontsize{11.38}{13.66}\selectfont 8}}
\pgfputat{\pgfxy(22.00,19.00)}{\pgfbox[bottom,left]{\fontsize{11.38}{13.66}\selectfont 7}}
\pgfputat{\pgfxy(52.00,19.00)}{\pgfbox[bottom,left]{\fontsize{11.38}{13.66}\selectfont 11}}
\pgfputat{\pgfxy(52.00,29.00)}{\pgfbox[bottom,left]{\fontsize{11.38}{13.66}\selectfont 3}}
\pgfputat{\pgfxy(32.00,19.00)}{\pgfbox[bottom,left]{\fontsize{11.38}{13.66}\selectfont 10}}
\pgfputat{\pgfxy(72.00,19.00)}{\pgfbox[bottom,left]{\fontsize{11.38}{13.66}\selectfont 5}}
\pgfputat{\pgfxy(52.00,39.00)}{\pgfbox[bottom,left]{\fontsize{11.38}{13.66}\selectfont 4}}
\pgfputat{\pgfxy(62.00,29.00)}{\pgfbox[bottom,left]{\fontsize{11.38}{13.66}\selectfont 6}}
\pgfputat{\pgfxy(43.00,49.00)}{\pgfbox[bottom,left]{\fontsize{11.38}{13.66}\selectfont 9}}
\pgfputat{\pgfxy(32.00,29.00)}{\pgfbox[bottom,left]{\fontsize{11.38}{13.66}\selectfont 1}}
\pgfputat{\pgfxy(42.00,39.00)}{\pgfbox[bottom,left]{\fontsize{11.38}{13.66}\selectfont 2}}
\pgfputat{\pgfxy(20.00,49.00)}{\pgfbox[bottom,left]{\fontsize{11.38}{13.66}\selectfont $T$}}
\pgfsetlinewidth{1.20mm}\pgfmoveto{\pgfxy(80.00,35.00)}\pgflineto{\pgfxy(90.00,35.00)}\pgfstroke
\pgfmoveto{\pgfxy(90.00,35.00)}\pgflineto{\pgfxy(87.20,35.70)}\pgflineto{\pgfxy(87.20,34.30)}\pgflineto{\pgfxy(90.00,35.00)}\pgfclosepath\pgffill
\pgfmoveto{\pgfxy(90.00,35.00)}\pgflineto{\pgfxy(87.20,35.70)}\pgflineto{\pgfxy(87.20,34.30)}\pgflineto{\pgfxy(90.00,35.00)}\pgfclosepath\pgfstroke
\pgfsetlinewidth{0.30mm}\pgfcircle[stroke]{\pgfxy(40.00,50.00)}{2.00mm}
\pgfcircle[fill]{\pgfxy(115.00,50.00)}{1.00mm}
\pgfcircle[stroke]{\pgfxy(115.00,50.00)}{1.00mm}
\pgfcircle[fill]{\pgfxy(105.00,40.00)}{1.00mm}
\pgfcircle[stroke]{\pgfxy(105.00,40.00)}{1.00mm}
\pgfcircle[fill]{\pgfxy(115.00,40.00)}{1.00mm}
\pgfcircle[stroke]{\pgfxy(115.00,40.00)}{1.00mm}
\pgfcircle[fill]{\pgfxy(125.00,40.00)}{1.00mm}
\pgfcircle[stroke]{\pgfxy(125.00,40.00)}{1.00mm}
\pgfcircle[fill]{\pgfxy(125.00,30.00)}{1.00mm}
\pgfcircle[stroke]{\pgfxy(125.00,30.00)}{1.00mm}
\pgfcircle[fill]{\pgfxy(135.00,30.00)}{1.00mm}
\pgfcircle[stroke]{\pgfxy(135.00,30.00)}{1.00mm}
\pgfcircle[fill]{\pgfxy(105.00,30.00)}{1.00mm}
\pgfcircle[stroke]{\pgfxy(105.00,30.00)}{1.00mm}
\pgfcircle[fill]{\pgfxy(105.00,20.00)}{1.00mm}
\pgfcircle[stroke]{\pgfxy(105.00,20.00)}{1.00mm}
\pgfcircle[fill]{\pgfxy(95.00,20.00)}{1.00mm}
\pgfcircle[stroke]{\pgfxy(95.00,20.00)}{1.00mm}
\pgfmoveto{\pgfxy(115.00,50.00)}\pgflineto{\pgfxy(115.00,40.00)}\pgfstroke
\pgfmoveto{\pgfxy(115.00,50.00)}\pgflineto{\pgfxy(125.00,40.00)}\pgfstroke
\pgfmoveto{\pgfxy(125.00,40.00)}\pgflineto{\pgfxy(135.00,30.00)}\pgfstroke
\pgfmoveto{\pgfxy(125.00,40.00)}\pgflineto{\pgfxy(125.00,30.00)}\pgfstroke
\pgfmoveto{\pgfxy(115.00,50.00)}\pgflineto{\pgfxy(105.00,40.00)}\pgfstroke
\pgfmoveto{\pgfxy(105.00,40.00)}\pgflineto{\pgfxy(105.00,30.00)}\pgfstroke
\pgfmoveto{\pgfxy(105.00,30.00)}\pgflineto{\pgfxy(105.00,20.00)}\pgfstroke
\pgfmoveto{\pgfxy(105.00,30.00)}\pgflineto{\pgfxy(95.00,20.00)}\pgfstroke
\pgfputat{\pgfxy(107.00,39.00)}{\pgfbox[bottom,left]{\fontsize{11.38}{13.66}\selectfont 8}}
\pgfputat{\pgfxy(97.00,19.00)}{\pgfbox[bottom,left]{\fontsize{11.38}{13.66}\selectfont 7}}
\pgfputat{\pgfxy(127.00,29.00)}{\pgfbox[bottom,left]{\fontsize{11.38}{13.66}\selectfont 3}}
\pgfputat{\pgfxy(107.00,19.00)}{\pgfbox[bottom,left]{\fontsize{11.38}{13.66}\selectfont 10}}
\pgfputat{\pgfxy(127.00,39.00)}{\pgfbox[bottom,left]{\fontsize{11.38}{13.66}\selectfont 4}}
\pgfputat{\pgfxy(137.00,29.00)}{\pgfbox[bottom,left]{\fontsize{11.38}{13.66}\selectfont 6}}
\pgfputat{\pgfxy(118.00,49.00)}{\pgfbox[bottom,left]{\fontsize{11.38}{13.66}\selectfont 9}}
\pgfputat{\pgfxy(107.00,29.00)}{\pgfbox[bottom,left]{\fontsize{11.38}{13.66}\selectfont 1}}
\pgfputat{\pgfxy(117.00,39.00)}{\pgfbox[bottom,left]{\fontsize{11.38}{13.66}\selectfont 2}}
\pgfputat{\pgfxy(95.00,49.00)}{\pgfbox[bottom,left]{\fontsize{11.38}{13.66}\selectfont $Y$}}
\pgfcircle[stroke]{\pgfxy(115.00,50.00)}{2.00mm}
\pgfcircle[fill]{\pgfxy(135.00,6.00)}{1.00mm}
\pgfcircle[stroke]{\pgfxy(135.00,6.00)}{1.00mm}
\pgfcircle[fill]{\pgfxy(125.00,-4.00)}{1.00mm}
\pgfcircle[stroke]{\pgfxy(125.00,-4.00)}{1.00mm}
\pgfcircle[fill]{\pgfxy(145.00,-4.00)}{1.00mm}
\pgfcircle[stroke]{\pgfxy(145.00,-4.00)}{1.00mm}
\pgfcircle[fill]{\pgfxy(120.00,6.00)}{1.00mm}
\pgfcircle[stroke]{\pgfxy(120.00,6.00)}{1.00mm}
\pgfcircle[fill]{\pgfxy(106.00,6.00)}{1.00mm}
\pgfcircle[stroke]{\pgfxy(106.00,6.00)}{1.00mm}
\pgfmoveto{\pgfxy(135.00,6.00)}\pgflineto{\pgfxy(145.00,-4.00)}\pgfstroke
\pgfmoveto{\pgfxy(135.00,6.00)}\pgflineto{\pgfxy(125.00,-4.00)}\pgfstroke
\pgfputat{\pgfxy(109.00,5.00)}{\pgfbox[bottom,left]{\fontsize{11.38}{13.66}\selectfont 7}}
\pgfputat{\pgfxy(127.00,-5.00)}{\pgfbox[bottom,left]{\fontsize{11.38}{13.66}\selectfont 11}}
\pgfputat{\pgfxy(122.00,5.00)}{\pgfbox[bottom,left]{\fontsize{11.38}{13.66}\selectfont 10}}
\pgfputat{\pgfxy(147.00,-5.00)}{\pgfbox[bottom,left]{\fontsize{11.38}{13.66}\selectfont 5}}
\pgfputat{\pgfxy(138.00,5.00)}{\pgfbox[bottom,left]{\fontsize{11.38}{13.66}\selectfont 6}}
\pgfputat{\pgfxy(94.00,5.00)}{\pgfbox[bottom,left]{\fontsize{11.38}{13.66}\selectfont $Z$}}
\pgfcircle[stroke]{\pgfxy(106.00,6.00)}{2.00mm}
\pgfcircle[stroke]{\pgfxy(120.00,6.00)}{2.00mm}
\pgfcircle[stroke]{\pgfxy(135.00,6.00)}{2.00mm}
\pgfsetlinewidth{1.20mm}\pgfmoveto{\pgfxy(80.00,20.00)}\pgflineto{\pgfxy(90.00,14.00)}\pgfstroke
\pgfmoveto{\pgfxy(90.00,14.00)}\pgflineto{\pgfxy(87.96,16.04)}\pgflineto{\pgfxy(87.24,14.84)}\pgflineto{\pgfxy(90.00,14.00)}\pgfclosepath\pgffill
\pgfmoveto{\pgfxy(90.00,14.00)}\pgflineto{\pgfxy(87.96,16.04)}\pgflineto{\pgfxy(87.24,14.84)}\pgflineto{\pgfxy(90.00,14.00)}\pgfclosepath\pgfstroke
\end{pgfpicture}%
\else
  \setlength{\unitlength}{1bp}%
  \begin{picture}(379.84, 175.75)(0,0)
  \put(0,0){\includegraphics{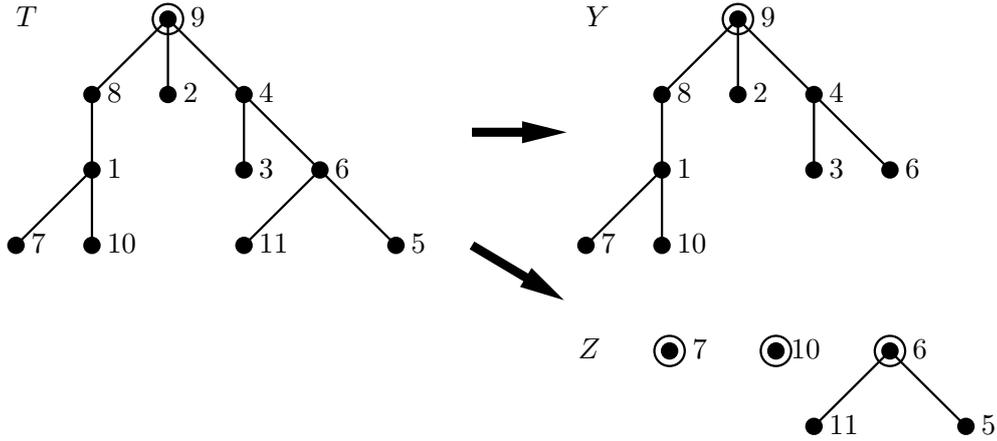}}
  \put(42.52,132.84){\fontsize{11.38}{13.66}\selectfont 8}
  \put(14.17,76.15){\fontsize{11.38}{13.66}\selectfont 7}
  \put(99.21,76.15){\fontsize{11.38}{13.66}\selectfont 11}
  \put(99.21,104.50){\fontsize{11.38}{13.66}\selectfont 3}
  \put(42.52,76.15){\fontsize{11.38}{13.66}\selectfont 10}
  \put(155.91,76.15){\fontsize{11.38}{13.66}\selectfont 5}
  \put(99.21,132.84){\fontsize{11.38}{13.66}\selectfont 4}
  \put(127.56,104.50){\fontsize{11.38}{13.66}\selectfont 6}
  \put(73.70,161.19){\fontsize{11.38}{13.66}\selectfont 9}
  \put(42.52,104.50){\fontsize{11.38}{13.66}\selectfont 1}
  \put(70.87,132.84){\fontsize{11.38}{13.66}\selectfont 2}
  \put(8.50,161.19){\fontsize{11.38}{13.66}\selectfont $T$}
  \put(255.12,132.84){\fontsize{11.38}{13.66}\selectfont 8}
  \put(226.77,76.15){\fontsize{11.38}{13.66}\selectfont 7}
  \put(311.81,104.50){\fontsize{11.38}{13.66}\selectfont 3}
  \put(255.12,76.15){\fontsize{11.38}{13.66}\selectfont 10}
  \put(311.81,132.84){\fontsize{11.38}{13.66}\selectfont 4}
  \put(340.16,104.50){\fontsize{11.38}{13.66}\selectfont 6}
  \put(286.30,161.19){\fontsize{11.38}{13.66}\selectfont 9}
  \put(255.12,104.50){\fontsize{11.38}{13.66}\selectfont 1}
  \put(283.46,132.84){\fontsize{11.38}{13.66}\selectfont 2}
  \put(221.10,161.19){\fontsize{11.38}{13.66}\selectfont $Y$}
  \put(260.79,36.46){\fontsize{11.38}{13.66}\selectfont 7}
  \put(311.81,8.12){\fontsize{11.38}{13.66}\selectfont 11}
  \put(297.64,36.46){\fontsize{11.38}{13.66}\selectfont 10}
  \put(368.50,8.12){\fontsize{11.38}{13.66}\selectfont 5}
  \put(342.99,36.46){\fontsize{11.38}{13.66}\selectfont 6}
  \put(218.27,36.46){\fontsize{11.38}{13.66}\selectfont $Z$}
  \end{picture}%
\fi
\caption{Decomposition of $T$ into $Y$ and $Z$}
\label{fig:TYZ}
\seo{여기에 새 그림}
\end{figure}

Now let us count the number of $p$-ary trees $T \in \Tp_{n,k}$ with $\abs{V(Y)}=m$ where $V(Y)$ is the set of vertices in $Y$.
First of all, the number of ways for selecting a set $V(Y) \subset [n]$  is equal to $n \choose m$.
By attaching $(m-k)$ increasing leaves to a decreasing $p$-ary tree with $k$ vertices, we can make a $p$-ary trees on $V(Y)$. So there are exactly $y(m,k)$ ways for making such a $p$-ary subtree on $V(Y)$.
Since all the roots of $Z$ are determined by $Y$, by the definition of $\F_{n,k}$ and Lemma~\ref{lem:fnk}, the number of ways for constructing the other parts on $V(T)\setminus V(\MD(T))$ is equal to
$$\left. f(n-k, m-k) \middle/  {n-k \choose m-k} \right. =  \frac{m-k}{n-k} (pn-pk)_{(n-m)}.$$
Since the range of $m$ is $k\le m \le n$, the equation \eqref{eq:num_o} holds.

Finally, $\Tp(n,n)$ is the set of decreasing $p$-ary trees on $[n]$, so
$$t(n,n) = y(n,n)=\prod_{j=0}^{n-1}\,(pj-j+1)$$
holds for $n \ge 1$.
\end{proof}
The sequence $t(n,k)$ with $p=2$ is listed in Table~\ref{table:tnk}. Note that each row sum is equal to $n!C^{(p)}_n$ with $p=2$.

\begin{table}[t]
{\footnotesize
\begin{tabular}{c|rrrrrrrrr|r}
$n\backslash k$ &0&1&2&3&4&5&6&7&8&$n!C_n$\\ \hline
0&1&&&&&&&&&1\\
1&0&1&&&&&&&&1\\
2&0&2&2&&&&&&&4\\
3&0&14&10&6&&&&&&30\\
4&0&152&104&56&24&&&&&336\\
5&0&2240&1504&816&360&120&&&&5040\\
6&0&41760&27744&15184&6992&2640&720&&&95040\\
7&0&942480&621936&342768&162240&65856&21840&5040&&2162160\\
8&0&24984960&16410240&9093888&4386304&1860224&680064&201600&40320&57657600\\
\end{tabular}
}
\vskip 1em
\caption{$t(n,k)$ with $p=2$}
\label{table:tnk}
\seo{삼각행렬 모습으로 보이는 것 중요할 듯. n=9,10을 생략하면 어떤지? 앞 테이블과 마찬가지로 upper triangulr 부분의 0 생략. 테이블과 켑션의 간격 조절 필요. $n!C_n$ 추가}
\end{table}



\begin{rmk}
Due to Lemma~\ref{lem:znk} and Theorem~\ref{thm:main}, we can calculate $t(n,k)$ for all $n$, $k$. In particular we express $t(n,k)$ as a linear combination of $y(k,k), y(k+1,k),\dots, y(n,k)$.
However a closed form, a recurrence relation, or a (double) generating function of $t(n,k)$ have not been found yet.
\seo{Remark가 Theorem 처럼 bold face이면 잘 보일 것 같음.}
\shin{적용 끝}
\end{rmk}

\section*{Acknowledgment}
For to the first authors,  this research was supported by Basic Science Research Program through the National Research Foundation of Korea (NRF) funded by the Ministry of Education (2011-0008683). For the second author, this research was supported by Basic Science Research Program through the National Research Foundation of Korea (NRF) funded by the Ministry of Science, ICT \& Future Planning (2012R1A1A1014154) and INHA UNIVERSITY Research Grant (INHA-44756).



\end{document}